\documentclass[12pt]{amsart}
\usepackage{amssymb,latexsym,amsfonts}
\usepackage{amsmath}
\usepackage{amscd}
\usepackage{pb-diagram,pb-xy}

\usepackage{graphicx}
\usepackage{color}
\usepackage[matrix,arrow]{xy}
\DeclareMathOperator*{\spin}{spin}

\newcommand{\Scal}{{\rm Scal}}

\newcommand{\Ric}{{\rm Ric}}

\newcommand{\str}{{\rm string}}
\DeclareMathAlphabet{\mathpzc}{OT1}{pzc}{m}{it}

\def\P{\mathbf P}
\def\H{\mathbf H}

\def\Z{{\mathbf Z}}

\newtheorem{theorem}{Theorem}[section]
\newtheorem{corollary}[theorem]{Corollary}
\newtheorem{lemma}[theorem]{Lemma}

\newtheorem{proposition}[theorem]{Proposition}
\newtheorem*{thm-main}{Main Theorem}
\newtheorem*{thm-a}{Theorem A}
\newtheorem*{thm-a1}{Corollary A}
\newtheorem*{thm-b}{Theorem B}
\newtheorem*{thm-b1}{Corollary B}
\newtheorem*{thm-c}{Theorem C}
\newtheorem*{thm-d}{Theorem D}
\newtheorem*{conj-d}{Conjecture D}
\newtheorem*{conj-c}{Conjecture C}
\newtheorem{thm}{Theorem}[section]
\theoremstyle{definition}
\newtheorem{definition}[thm]{Definition}

\newtheorem{Nremark}[thm]{Remark}
\newtheorem*{remark}{Remark}

\begin{document}

\author{Boris Botvinnik} 
\address{Department of Mathematics
  University of Oregon
  Eugene OR 97403-1222, U.S.A.}
\email{botvinn@math.uoregon.edu}

\author{Mohammed Labbi} \address{Department of Mathematics College of
  Science University of Bahrain 32038, Bahrain.}
\email{mlabbi@uob.edu.bh} \title{Compact manifolds with positive
  $\Gamma_2$-curvature} \date{}  \thanks{The authors would like to thank S. Hoelzel
  for early communication of his results and for useful critical
  comments on the previous version of this paper.}
 \maketitle
\begin{abstract} 
The Schouten tensor \ $A$ \ of a Riemannian manifold \ $(M,g)$
\ provides important scalar curvature invariants
$\sigma_k$,  that are the  symmetric functions on the eigenvalues of
$A$, where, in particular, $\sigma_1$ \ coincides with the standard
scalar curvature \ $\Scal(g)$. Our goal here is to study compact
manifolds with positive \ $\Gamma_2$-curvature, \ i.e.,
when $\sigma_1(g)>0$ and $\sigma_2(g)>0$.  In particular, we prove
that a 3-connected non-string manifold $M$ admits a positive
$\Gamma_2$-curvature metric if and only if it admits a positive
scalar curvature metric. Also we show that any
finitely presented group $\pi$ can always be realised
as the fundamental group of a closed manifold of positive
$\Gamma_2$-curvature and of arbitrary dimension greater
than or equal to six.
\end{abstract}

\section{Introduction and statement of the results}
\subsection{Motivation}
Let $(M,g)$ be a Riemannian manifold of dimension $n\geq 3$. Recall
that the Riemann curvature tensor $R$ of $(M,g)$ decomposes into
$$
R=W+gA,
$$
where $W$ is the trace-free part of $R$, that
is, the Weyl tensor, and the product $gA$ is the
Kulkarni-Nomizu product of the metric $g$ and the
Schouten tensor $A$. The latter is defined by
$$ 
A=\frac{1}{n-2}\left(\Ric -\frac{\Scal}{2(n-1)}g\right).
$$
Here $\Ric$ and $\Scal$ are respectively the Ricci
curvature tensor and the scalar curvature of $(M,g)$.

Let $\lambda_1,\ldots,\lambda_n$ denote the eigenvalues of the
operator associated to $A$ via the metric $g$. For $1\leq k\leq n$,
the \emph{$\sigma_k$-curvature} of $(M,g)$ is defined to be the scalar
function
$$
\sigma_k=\sigma_k(A)=\sum_{1\leq i_1 <...<i_k\leq n}\lambda_{i_1}...\lambda_{i_k}.
$$ 
In this paper, we concentrate on the
$\sigma_2$-curvature; more specifically, we study its
positivity properties.  First we note that
$$
\sigma_1^2=2\sigma_2+\|A\|^2.
$$ In particular, the positivity of $\sigma_2$ implies that
$\sigma_1^2$ is never zero. This forces $\sigma_1$ to
have a constant sign on $(M,g)$.  In other words, the
condition $\sigma_2>0$ has two possibilities:
$\sigma_2>0$ and $\sigma_1>0$ or $\sigma_2>0$ and $\sigma_1 <0$.
\begin{definition}
We say that $(M,g)$ has \emph{positive
  $\Gamma_2$-curvature} if $\sigma_2>0$ and $\sigma_1>0$.
\end{definition}
Here is the main question we would like to resolve:

{\bf Question 1:} Which manifolds admit metrics with positive
$\Gamma_2$-curvature?

\subsection{Main results on the existence of positive 
$\Gamma_2$-curvature} It is important to emphasize that the positivity
of $\Gamma_2$-curvature originally appeared as an ellipticity
assumption that ensures that the $\sigma_2$-Yamabe equation is
elliptic at any solution; see for instance \cite{GVW, Viac}.  There
are deep and interesting results concerning the $\sigma_2$-Yamabe
problem, which inspired great interest from the mathematics community
in the curvatures associated with the Shouten tensor.

Our first main result is an affirmative answer to Question 1 under
certain topological restrictions on $M$. Let $p_k(M)$
stand for the $k$-th Pontryagin class of the manifold $M$.  We recall
that a smooth spin manifold $M$ is a \emph{string
  manifold} if $\frac{1}{2}p_1(M)=0$. Otherwise, we
say that $M$ is \emph{not string}.
\begin{thm-a} 
  Let $M$ be a compact $3$-connected non-string manifold with $\dim
  M=n\geq 9$ which admits a metric of positive scalar curvature. Then
  $M$ admits a Riemannian metric $g$ with positive
  $\Gamma_2$-curvature.
\end{thm-a}
We notice that such a manifold is spin and determines a cobordism
class $[M]\in \Omega^{\spin}_n$. In particular, the index map $\alpha:
\Omega^{\spin}_n\to KO_n$ gives an element $\alpha([M])\in
KO_n$. Then, according to Gromov-Lawson \cite{GroLaw} and Stolz
\cite{Stolz1}, such a manifold $M$ admits a positive scalar curvature
metric if and only if $\alpha([M])=0$.  Thus, Theorem
A provides an affirmative answer to the existence question as follows:
\begin{thm-a1}
A compact $3$-connected non-string manifold $M$ with
$\dim M=n\geq 9$ admits a metric of positive $\Gamma_2$-curvature if
and only if $\alpha([M])=0$ in $KO_n$.
\end{thm-a1}
In the case when a manifold $M$ is string, we have
  the following  result:
\begin{thm-b}
A compact $3$-connected string manifold $M$ of dimension $n\geq 9$
that is string cobordant to a manifold of positive
$\Gamma_2$-curvature admits a metric with positive
$\Gamma_2$-curvature.
\end{thm-b}
We recall that a string manifold $M$ determines a cobordism class
$[M]\in \Omega^{\str}_n$. We denote by $\phi_W : \Omega^{\str}_*\to
\Z[[q]]$, the Witten genus; see \cite{Dessai,Stolz1}. We prove the
following result which is analogous to \cite[Corollary B]{GroLaw}.
\begin{thm-b1}\label{main-2}
Let $M$ be a $3$-connected string manifold of
dimension $n\geq 9$. Assume $\phi_W([M])=0$. Then some multiple
$M\# \cdots \# M$ carries a metric of positive
$\Gamma_2$-curvature.
\end{thm-b1}
It is well-known that some curvature conditions (such as positivity of
Ricci curvature) impose severe restrictions on the fundamental group
of a manifold. We show that the positivity of $\Gamma_2$-curvature
does not require any restrictions. We prove the following:
\begin{thm-c}\label{fund-group} Let $\pi$ be a finitely presented group.
Then for every $n\geq 6$, there exists a compact $n$-manifold M with
 positive $\Gamma_2$-curvature such that
$\pi_1(M)=\pi$.
\end{thm-c}
\subsection{Plan of the paper}
In Section \ref{sec:positivity}, we describe close relationships
between positivity of the Einstein tensor and other curvatures and
positivity of the $\Gamma_2$-curvature, which provides relevant
techniques and ideas to prove the above results. In Section
\ref{sec:submersion}, we examine the $\Gamma_2$-curvature on a
Riemannian submersion. In particular, we show that a geometric
$\H\P^2$-bundle carries a metric with positive $\Gamma_2$-curvature.
We prove the main results in Section \ref{sec:proofs}. In Section
\ref{sec:general}, we describe a general approach to the stability of some curvature conditions under surgery
developed by S. Hoelzel \cite{Hoelzel} and we derive some
applications. In particular, we show that any finitely presented group
can be realized as a fundamental group of an $n$-manifold carrying a
metric with positive $\Gamma_k$-curvature, provided that $2\leq
k<\frac{n+1}{2}-\frac{1}{2}\sqrt{n-\frac{1}{n-1}}$.

\section{Positivity of $\Gamma_2$ and related curvatures}\label{sec:positivity}
\subsection{Positive $\Gamma_2$-curvature and positive Einstein curvature}
Here, we will explain that the ellipticity condition
for the $\sigma_2$-Yamabe equation is closely related
to the positive definiteness of the Einstein tensor
\[
S=\frac{\Scal}{2}g-\Ric.
\]
We notice first that from an algebraic view point,
the Einstein tensor is the first Newton transformation of Schouten
tensor $A$ as follows:
\begin{lemma}
The Einstein tensor $S$ is the image of the Schouten tensor $A$ under
the first Newton transformation: 
\begin{equation}
S=(n-2)\left( \sigma_1(A)g-A\right).
\end{equation}
\end{lemma}
\begin{proof} Indeed, we have:
\begin{equation*}
 \sigma_1(A)g-A=\frac{\Scal}{2(n-1)}g-\frac{1}{n-2}\left(\Ric
 -\frac{\Scal}{2(n-1)}g\right)=\frac{1}{n-2}\left(
 \frac{\Scal}{2}g-\Ric\right).
\end{equation*}
This proves it.
\end{proof}
The next proposition shows that positivity of $\Gamma_2$-curvature
implies positive definiteness of the Einstein tensor:
\begin{proposition}\label{weaker}
Let $n\geq 3$. Assume that a Riemannian $n$-manifold $(M,g)$ has
positive $\Gamma_2$-curvature. Then its Einstein tensor $S$ is
positive definite.  In particular, if $n=3$, positive
$\Gamma_2$-curvature implies positive sectional curvature.
\end{proposition}
\begin{remark}
The previous proposition is a special case of a general result, see
Lemma \ref{CNS1} below and \cite[Proposition 1.1]{CNS}.
Below, we provide a direct proof of Proposition
\ref{weaker}. In particular, we obtain a lower bound for the Einstein
tensor. This proof is inspired by the proof of a similar result in
dimension $4$ in \cite{CGY}.
\end{remark}
\begin{proof}[Proof of Proposition \ref{weaker}]
We show that if $\sigma_1(A)>0$ and $\sigma_2(A)>0$, then the Einstein
tensor is positive definite.  Assume that $\sigma_1(A)>0$ and write
$A=A_1+\frac{\sigma_1(A)}{n}g$ where $A_1=A-\frac{\sigma_1(A)}{n}g$ is
a trace-free tensor.  Then the first Newton transformation is given as
$$
\begin{array}{c}
t_1(A)=\sigma_1(A)g-A=\frac{n-1}{n}\sigma_1g-A_1.
\end{array}
$$ 
Now let $X$ be a unit vector. Then we have:
$$ 
\begin{array}{c}
t_1(A)(X,X)=\frac{n-1}{n}\sigma_1-A_1(X,X)\geq
\frac{n-1}{n}\sigma_1-|A_1(X,X)|.
\end{array}
$$ 
Next, since $A_1$ is trace free, a simple Lagrange multiplier
argument shows that
\[
\begin{array}{c}
|A_1(X,X)|\leq \frac{\sqrt{n-1}}{\sqrt{n}}||A_1||.
\end{array}
\]
Therefore, we obtain
\begin{equation*}
\begin{array}{lcl}
 t_1(A)(X,X)&\geq &
 \frac{n-1}{n}\sigma_1-\frac{\sqrt{n-1}}{\sqrt{n}}||A_1|| \\ \\ &\geq
 & \frac{n-1}{n}\sigma_1-2\left(
 \frac{||A_1||}{\sqrt{2}\sqrt{\sigma_1}}\frac{\sqrt{\sigma_1}\sqrt{n-1}}{\sqrt{2}\sqrt{n}}\right)
 \\ \\ &\geq & \frac{n-1}{n}\sigma_1-
 \frac{||A_1||^2}{2\sigma_1}-\frac{\sigma_(n-1)}{2n} \\ \\ & = &
 \frac{n-1}{2n}\sigma_1- \frac{||A_1||^2}{2\sigma_1} \\ \\ &=&
 \bigl(\frac{n-1}{n}\sigma_1^2-||A_1||^2\bigr)\frac{1}{2\sigma_1}=
\frac{2\sigma_2}{2\sigma_1}=\frac{\sigma_2}{\sigma_1}.
\end{array}
\end{equation*}
Thus, $ t_1(A)\geq \frac{\sigma_2}{\sigma_1}$.
\end{proof}
As a direct consequence of the previous proof, we
have:
\begin{proposition}
Let $(M,g)$ be a Riemannian manifold of dimension
$n\geq 3$ and positive scalar curvature. Then the
  first eigenvalue of the Einstein tensor is bounded from below by
  $(n-2)\frac{\sigma_2}{\sigma_1}$.
\end{proposition}
On the other hand, positivity of the Einstein tensor does not imply
positivity of the $\Gamma_2$-curvature.  Moreover, there are examples
of metrics with positive Einstein tensor that cannot be conformally
deformed to a metric with positive $\Gamma_2$-curvature.
\vspace{5mm}

Let $n\geq 2$, $(S^{n+2},g_0)$ be the standard sphere of curvature
$+1$, and $(M^n,g_1)$ be a compact space form of curvature $-1$.  We
define $g$ to be the Riemannian product of $g_0$ and $g_1$ on $S^{n+2}
\times M^n$.
 \begin{proposition}
Let $(S^{n+2} \times M^n,g)$ be as above with $n\geq 2$. Then 
\begin{enumerate}
\item[{\bf (1)}] $\Scal(g)>0$ and the Einstein tensor $S(g)$ is
  positive definite;
\vspace{1mm}

\item[{\bf (2)}] there is no metric $\bar g$ in the conformal class
  $[g]$ with positive $\Gamma_2$-curvature.
\end{enumerate}
\end{proposition}
\begin{proof}
It is not difficult to check that $\Scal(g)>0$ and the Einstein tensor
$S(g)$ is positive definite.  The metric $g$ is conformally flat, and
therefore, any metric in the conformal class of $g$ is conformally
flat as well. Then \cite[Theorem 1.1]{GLW} asserts that the existence
of a conformally flat metric with positive $\Gamma_2$-curvature on a
compact manifold of dimension $2n+2$ implies the vanishing of its
$n$-th Betti number. However, the $n$-th Betti number of $S^{n+2}
\times M^n$ is the same one as for $M^n$, which cannot be equal to
zero.
\end{proof}
Thus we see that positivity of the
$\Gamma_2$-curvature of a Riemannian metric is
a rather strong condition since it
implies positivity of the Einstein tensor in all dimensions at least
three. In dimension $3$, it
guarantees positivity of sectional curvature, and in dimension
$4$, it implies positivity of Ricci curvature.

On the other hand, as we shall show in the sequel to
this paper, there are very simple examples, such as the product of
spheres $S^3\times S^{q}$, with $q\geq 2$, which admit a Riemannian
metric that has non-negative sectional curvature,
positive Ricci curvature and positive definite Einstein tensor,
but its
$\sigma_2$-curvature is negative. We
address here the following natural question:

{\bf Question 2.} Which curvature conditions imply positive
$\Gamma_2$-curvature?

Clearly, positivity of the Schouten tensor implies positivity of
$\Gamma_2$-curvature, however, this condition is too strong. In fact,
positivity of the Schouten tensor is equivalent to positivity of all
$\sigma_k$-curvatures for $1\leq k\leq n$.  In the next
proposition, we give much weaker conditions
that imply positivity of $\Gamma_2$-curvature:
\begin{proposition}
Let $n=2k+2$ be even and $(M,g)$ be a Riemannian $n$-manifold such
that the sum of the smallest $k+1$ eigenvalues of the Schouten tensor
is positive. Then the metric $g$ has positive $\Gamma_2$-curvature.
\end{proposition}
\begin{proof}
One can easily prove the previous result using the
exterior product of double forms as in
\cite{Lab-min,Lab-ident}. Using the notations of these two
references, we get
$$
\begin{array}{c}
\sigma_2=\ast\frac{1}{(n-2)!}g^{n-2}A^2=\ast\frac{1}{(n-2)!}\left(
g^kAg^kA\right)=\frac{1}{(n-2)!}\langle \ast g^kA,g^kA\rangle.
\end{array}
$$ 
The last inner product can be easily seen as a sum whose terms are
products of two factors. One factor is the sum of $(k+1)$ eigenvalues
of $A$ and the other factor is the sum of the remaning $(k+1)$
eigenvalues of $A$. This shows that $\sigma_2$ is
positive. To be more explicit, let $e_1,...,e_n$ be an orthonormal
basis that diagonalises $A$. Then
$$
\langle \ast g^kA,g^kA\rangle=\sum_{1\leq i_1 <...<i_{k+1}\leq n}
\left( \lambda_{i_1}+...+\lambda_{i_{k+1}} \right) \left( \sigma_1-
\lambda_{i_1}-...- \lambda_{i_{k+1}}\right) .
$$
Clearly, $\sigma_1>0$. This completes the proof.
\end{proof}
\subsection{Positive $\sigma_2$-curvature and 
positive second  Gauss-Bonnet curvature} We recall that the
  \emph{second Gauss-Bonnet curvature} $h_4$ is defined by
\begin{equation}\label{equation1}
h_4=\|R\|^2-\|\Ric\|^2+\frac{1}{4}\Scal^2.
\end{equation}
The positivity properties of $h_4$ were studied in
\cite{Labbisgbc,Bot-Lab}. It turns out that the $\sigma_2$-curvature
coincides (up to a constant factor) with $h_4$ for conformally-flat
manifolds. Precisely, the $\sigma_2$-curvature is the non-Weyl part of
$h_4$ as follows:
\begin{proposition}[\cite{Labbih4Y}]
For an arbitrary Riemannian manifold of dimension $n\geq
4$, we have
\begin{equation*}
h_4=|W|^2+2(n-2)(n-3)\sigma_2.
\end{equation*}
In particular, positive (resp. nonnegative) $\sigma_2$-curvature
implies $h_4> 0$ (resp. $h_4 \geq 0$).
\end{proposition}
\subsection{Higher $\sigma_k$-curvatures and higher 
Gauss-Bonnet curvatures} Even though this paper is mainly about the
$\Gamma_2$-curvature (i.e., scalar and
$\sigma_2$-curvature), we briefly discuss here the higher
$\sigma_k$-curvatures and corresponding $\Gamma_k$-curvatures.  In
this subsection we shall use the formalism of double forms as in
\cite{Lab-min,Lab-ident}.  Recall that the $\sigma_k$-curvature and
the Newton transformation $t_k(A)$ of a Riemannian manifold $(M,g)$
are given by:
\begin{equation}
\begin{array}{lclclclll}
\sigma_k &=&\sigma_k(A)&=& \displaystyle
\ast\frac{g^{n-k}A^k}{(n-k)!k!}&=&\displaystyle
\frac{c^kA^k}{(k!)^2} & \mathrm{for} & \,\, 0\leq k\leq n.
\\
\\
t_k&=&t_k(A)&=&\displaystyle
\ast\frac{g^{n-k-1}A^k}{(n-k-1)!k!}&=&\displaystyle
\sigma_k(A)g-\frac{c^{k-1}A^k}{(k-1)!k!} &\mathrm{for} & \,\, 1\leq k\leq n.
\end{array}
\end{equation}
\begin{remark}
Here we considered the Schouten tensor  $A$ as a $(1,1)$-double form. 
\end{remark}
The following algebraic fact is a classical
result, see \cite[Proposition 1.1]{CNS}:
\begin{lemma}\label{CNS1}
If for all $r$ such that $1\leq r\leq k+1\leq n$ we have
$\sigma_r(A)>0$ then the tensor $t_k(A)$ is positive definite.
\end{lemma}
We are now going to apply Lemma \ref{CNS1} to the higher Gauss-Bonnet
curvatures and Einstein-Lovelock tensors.  First,
recall that for $2\leq 2k\leq n$, the $2k$-th Gauss-Bonnet curvature
$h_{2k}$ of $(M,g)$ is the function defined on $M$ by
\begin{equation}\label{D}
h_{2k}={1\over (n-2k)!}*\bigl( g^{n-2k}R^k\bigr).
\end{equation}
 The $2k$-th Einstein-Lovelock tensor of $(M,g)$, denoted $T_{2k}$, is
 defined by
\begin{equation}\label{D+}
T_{2k}=*{1\over (n-2k-1)!}g^{n-2k-1}R^k.
\end{equation}
If $2k=n$, we set $T_n=0$. For $k=0$, we set $h_0=1$
and $T_0=g$.

Suppose now that $(M,g)$ is a conformally flat
manifold. Then its Riemann curvature tensor is
determined by the Schouten tensor $A$ as follows:
$$
R=gA.
$$
Consequently, for $n-2k-1\geq 0$, we have:
\begin{equation}
t_k(A)\!=
\!\ast\frac{g^{n-k-1}A^k}{(n-k-1)!k!}=
\ast\frac{g^{n-2k-1}(gA)^k}{(n-k-1)!k!}=
\ast\frac{g^{n-2k-1}R^k}{(n-k-1)!k!}=\frac{(n-2k-1)!}{k!(n-k-1)!}T_{2k}.
\end{equation} 
Similarly, we have for $n\geq 2k$:
\begin{equation}
\sigma_k(A) =\ast\frac{g^{n-k}A^k}{(n-k)!k!}
=\ast\frac{g^{n-2k}(gA)^k}{(n-k)!k!}=\ast\frac{g^{n-2k}(R)^k}{(n-k)!k!}
=\frac{(n-2k)!}{(n-k)!k!}h_{2k}.
\end{equation}
We have therefore proved the following result:
\begin{proposition}\label{conf-flat}
For   a conformally flat $n$-manifold $(M,g)$,  the following are true
\begin{enumerate}
\item For $2\leq 2k\leq n$, the $\sigma_k$-curvature is positive if and only if  the $h_{2k}$-Gauss-Bonnet curvature is positive.
\item If the Gauss-Bonnet curvatures $h_2,h_4,...,h_{2k+2}$ of $(M,g)$
  are all positive then the Einstein-Lovelock tensors $T_2,...,T_{2k}$
  of $(M,g)$ are all positive definite.
\end{enumerate}
\end{proposition}
\section{Riemannian submersions and $\sigma_2$-curvature}\label{sec:submersion}
\subsection{General observations}
Let $(M,g)$ be a total space of a Riemannian submersion $p: M\to
B$. We denote by $\hat g$ the restriction of $g$ to
the fiber $F=p^{-1}(x)$, $x\in B$. We denote by $\Scal(\hat g)$ and
$\sigma_2(\hat g)$ the scalar curvature and $\sigma_2$-curvature of
the metric $\hat g$, respectively.

For a Riemannian submersion $p: M\to B$ as above, there is a
\emph{canonical variation} $g_t$ of the original metric $g$, which is
a fiberwise scaling by $t^2$. Then, if the fiber metrics $\hat g$
satisfies some bounds, this construction delivers a metric on a total
space with positive (negative) curvature.
\begin{theorem}[See \cite{Labbih4Y}]\label{Thm1}
Let $p: M\to B$ be a Riemannian submersion where the total space
$(M,g)$ is a compact manifold with $\dim M =n$ and fibers
$F$, where $\dim F=p$. Let $\hat g$ be the induced
  metric on fiber $F$. Assume that the inequality
\begin{equation}\label{conditionA}
\begin{array}{llcl}
&8(n-1)(p-1)(p-2)^2\sigma_2(\hat g) &>& (n-p)\Scal^2(\hat g)
  \ \ \\ \\ \mbox{(resp.} & 8(n-1)(p-1)(p-2)^2\sigma_2(\hat
  g) &<& (n-p)\Scal^2(\hat g))
\end{array}
\end{equation}
holds for every fiber $(F,\hat g)$.  Then there exists $t_0>0$ such
that the canonical variation metric $g_t$ on $M$ has $\sigma_2(g_t)>0$
(resp. $\sigma_2(g_t)<0$) for all $0<t\leq t_0$.  

Furthermore, if $\Scal (\hat g)>0$ (resp. $\Scal (\hat
g)<0$), then $\Scal (g_t)>0$ (resp. $\Scal (g_t)<0$)
for all $0<t\leq t_0$.
\end{theorem}
\begin{proof} We modify  the proof of \cite[Theorem B]{Labbisgbc} 
a bit.  Using O'Neill's formulas for Riemannian
submersions, we get an estimate of $ \Scal^2(g_t)$
and the norm of the Ricci curvature $||\Ric(g_t)||_{g_t}^2$ of the
canonical variation metric $g_t$ as follows:
$$
\begin{array}{c}
||\Ric(g_t)||_{g_t}^2={1\over t^4}||\Ric(\hat g)||^2+O({1\over
  t^2}),\,\, \mbox{and}\,\, \Scal^2(g_t)={1\over t^4} \Scal^2(\hat
g)+O({1\over t^2});
\end{array}
$$ 
see \cite{Labbisgbc} for details.  Consequently, the
$\sigma_2$-curvature of the metric $g_t$ is given by
\begin{equation*}
\begin{array}{lcl}
2(n-2)^2\sigma_2(g_t)&=& -||\Ric(g_t)||_{g_t}^2+\frac{n}{4(n-1)} \Scal^2(g_t)
\\
\\
&=& \frac{1}{t^4}\big\{(2(p-2)^2\sigma_2(\hat g) -\frac{n-p}{4(n-1)(p-1)}\Scal^2(\hat g)\bigr\}+O({1\over t^2}).
\end{array}
\end{equation*}
To prove the second part of the theorem, it is enough to recall that
$$ 
\Scal (g_t)={1\over t^2} \Scal (\hat g)+O(1).
$$
This concludes the proof.
\end{proof}
If the fiber $(F,\hat g)$ is Einstein with $\dim F=p$, then its
$\sigma_2$-curvature is determined by the scalar
curvature as follows (see \cite[Proposition
  3.2]{Labbih4Y}):
\begin{equation}\label{sigma}
\sigma_2(\hat g) = \frac{1}{8p(p-1)}\Scal^2(\hat g).
\end{equation}
 Consequently, the $\sigma_2$-curvature of the metric $g_t$ of the
 total space is given by
\begin{equation}\label{eq3}
\begin{array}{c}
2(n-2)^2\sigma_2(g_t)= \frac{\Scal^2(\hat
  g)}{t^4}\left(\frac{n(p-4)+4}{p(n-1)}\right)+O\left({1\over
  t^2}\right).
\end{array}
\end{equation}
Therefore, the first inequality in (\ref{conditionA}) holds for such a
metric $\hat g$, provided that the dimension $\dim F=p\geq 4$ and
$\hat g$ is not Ricci flat. Thus, we have proved the
following result:
\begin{corollary}\label{cor:1-2}
Let $p: M\to B$ be a Riemannian submersion, where the total space
$(M,g)$ is a compact manifold with $\dim M =n$ and fibers
$F$, where $\dim F=p\geq 4$. Let $\hat g$ be the
induced metric on fiber $F$. Assume that the fiber metrics $\hat g$
are non Ricci-flat Einstein metrics. Then there exists $t_0>0$ such
that the canonical variation metric $g_t$ on $M$ has $\sigma_2(g_t)>0$
for all $0<t\leq t_0$.

Furthermore, the scalar curvature $\Scal(g_t)$ of the metric $g_t$ has
the same sign as the scalar curvature of the fiber metrics $\hat g$
for all $0<t\leq t_0$.
\end{corollary}
Let $\H\P^2$ be a quaternionic projective plane equipped with a
standard metric $\hat g_0$. It is well-known that the metric $\hat
g_0$ is Einstein with positive sectional curvature and that it has
isometry group $PSp(3)$.

Then a smooth fiber bundle $\pi: M\to B$ is called a \emph{geometric
  $\H\P^2$-bundle} if its fiber is $\H\P^2$ and the structure group is
$PSp(3)$. Given any Riemannian metric on the base $B$, there is a
canonical metric $g_0$ on the total space $M$ inducing the metric
$\hat g_0$ along every fiber, and the projection $\pi: M\to B$ becomes
a Riemannian submersion.  Thus, we have the following:
\begin{corollary}\label{HP2}
Let $\pi: M\to B$ be a geometric $\H\P^2$-bundle. Then $M$ admits a
metric with positive $\Gamma_2$-curvature.
\end{corollary}
\subsection{Negative $\sigma_2$-curvature for low-dimensional fibers}
If the fibers have lower dimensions, the above machinery helps to
construct metrics with negative $\sigma_2$-curvature as follows.
\begin{corollary}
Let $p: M\to B$ be a Riemannian submersion where the total space
$(M,g)$ is a compact manifold with $\dim M =n$ and fibers
$F$, where $\dim F=2$. Let $\hat g$ be the
induced metric on fiber $F$.  Assume that the Gaussian curvature of
the fiber metrics $\hat g$ does not vanish and $n\geq 3$.  Then there
exists $t_0>0$ such that the canonical variation metric $g_t$ on $M$
has $\sigma_2(g_t)<0$ for all $0<t\leq t_0$.

Furthermore, the scalar curvature $\Scal(g_t)$ of the metric $g_t$ has
the same sign as the Gaussian curvature of the fiber metrics $\hat g$
for all $0<t\leq t_0$.
\end{corollary}
In the case of three-dimensional fibres we
have:
\begin{corollary}
Let $p: M\to B$ be a Riemannian submersion where the total space
$(M,g)$ is a compact manifold with $\dim M =n\geq 5$ and fibers
$F$, where $\dim F=3$. Let $\hat g$ be the induced
metric on fiber $F$. Assume that the fiber metrics $\hat g$ are non
Ricci-flat Einstein metrics. Then there exists $t_0>0$ such that the
canonical variation metric $g_t$ on $M$ has $\sigma_2(g_t)<0$ for all
$0<t\leq t_0$.

Furthermore, the scalar curvature $\Scal(g_t)$ of the metric
$g_t$ has the same sign as the scalar curvature of the fiber metrics
$\hat g$ for all $0<t\leq t_0$.
\end{corollary}
Next, we specify the previous results to products with the standard
spheres.  Let $\hat g(r)$ be a standard round metric on the sphere
$S^p$ of radius $r$ .
\begin{corollary}\label{example}
Let $(M,g(r))=(S^p,\hat
  g(r)) \times (B,g_B)$, where $(B,g_B)$ is an arbitrary compact
  Riemannian manifold.
\begin{itemize}
\item If $p=2$  and $\dim B\geq 1$, then  $\Scal(g(r))>0$  and
$\sigma_2(g(r))<0$ for all  $r$ sufficiently small.

\vspace{2mm}

\item If  $p=3$ and $\dim B\geq 2$, then  $\Scal(g(r))>0$  and
$\sigma_2(g(r))<0$ for all  $r$ sufficiently small.

\vspace{2mm}

\item If $p\geq 4$, then  $\Scal(g(r))>0$ and
  $\sigma_2(g(r))>0$ for all  $r$ sufficiently small.

\end{itemize}
\end{corollary}
\begin{Nremark}
We notice that the Riemannain product $$(M, g(r))=(S^3(r),\hat g(r))\times
(S^{q},ds^2) $$ of the standard spheres (where $\hat g(r)$ is a round
metric of radius $r$ and $q\geq 2$) is such that for small enough
$r$, the sectional curvature of $g(r)$ is nonnegative, Ricci
curvature, the Einstein tensor and the $h_4$-curvature are all
positive, but its $\sigma_2$-curvature is negative.
\end{Nremark}
\begin{Nremark}
Recall that any finitely presented group can be realised as the
fundamental group of a compact manifold of an arbitrary dimension
$n\geq 4$. Consequently, the above examples show that
any finitely presented group can be realized as the
fundamental group of a compact $n$-manifold of positive scalar
curvature and, at the same time, of negative $\sigma_2$-curvature for
any arbitrary $n\geq 6$.

The same is true for compact $n$-manifolds of positive
$\Gamma_2$-curvature for $n\geq 8$. In Section \ref{fund:group}, we
will show that this is still true for $n\geq 6$.
\end{Nremark}

\section{Proofs of main theorems}\label{sec:proofs}
\subsection{A surgery theorem for metrics with positive 
$\Gamma_2$-curvature}\label{sec:surgery}
 Let $X$ be a closed manifold,
$\dim X =n$, and $S^p\subset X$ be an embedded sphere in $X$ with
trivial normal bundle. We assume that it is embedded together with its
tubular neighbourhood $S^p\times D^{q}\subset X$.  Here, $p+q =
n$. Then we define $X'$ to be the manifold resulting from the surgery
along the sphere $S^p$:
$$
X' = (X\setminus (S^p\times D^{q}))\cup_{S^p\times S^{q-1}}
(D^{p+1}\times S^{q-1}) .
$$ The codimension of the sphere $S^p\subset X$ is called a
\emph{codimension of the surgery}. In the above terms, the codimension
of the above surgery is $q$.
\begin{theorem}\label{surgery}
Let $X$ be a compact manifold with $\dim X\geq 5$ and $X'$ be a
manifold obtained from \ $X$ \ by a surgery of codimension at least
$5$. \ Assume that $X$ has a Riemannian metric $g$ with positive
$\Gamma_2$-curvature.  Then there exists a metric $g'$ on $X'$ with
positive $\Gamma_2$-curvature.
\end{theorem}
The above surgery theorem will be used  in the following subsections to prove the main theorems of this paper, its proof will be postponed to the next section.
\subsection{Fundamental groups of manifolds with  positive 
$\Gamma_2$-curvature}\label{fund:group} We already mentioned that
Corollary \ref{example} implies that there are no restrictions on the
fundamental group of a compact manifold of positive scalar curvature
and positive $\sigma_2$-curvature in dimensions at least eight. We use
surgery Theorem \ref{surgery} to prove that the same holds in
dimensions $6$ and $7$:
\begin{thm-c}\label{fund-group} Let $\pi$ be a finitely presented group.
Then for every $n\geq 6$, there exists a compact $ n$-manifold M with
positive $\Gamma_2$-curvature such that $\pi_1(M)=\pi$.
\end{thm-c}
\begin{remark}
We emphasize that the previous result is no longer true in dimensions
$4$ and $3$. In dimension $4$, the positivity of the
$\Gamma_2$-curvature implies the positivity of the Ricci
curvature; consequently, the
fundamental group must be finite. The same holds for the fundamental
group of a $3$-dimensional compact manifold of positive
$\Gamma_2$-curvature as in this case, the sectional
curvature must be positive. However, it remains an open question
whether there exist any restriction on the fundamental group of
$5$-dimensional compact manifold of positive $\Gamma_2$-curvature.
\end{remark}
\begin{proof}[Proof of Theorem C] 
Let $n\geq 6$ and $\pi$ be a group which has a presentation consisting
of $k$ generators $x_{1},x_{2},...,x_{k}$ and $\ell$ relations
$r_{1},r_{2},...,r_{\ell}$.  Let the manifold $S^{1}\times S^{n-1}$ be
given a standard product metric which has positive
$\Gamma_2$-curvature.  Since $\pi_1(S^{1}\times S^{n-1})\cong \Z$, the
Van-Kampen theorem implies that the fundamental group of the connected
sum
$$
N:= \#k (S^{1}\times S^{n-1})
$$ is a free group on $k$ generators, which we denote by
$x_{1},x_{2},...,x_{k}$. By the surgery Theorem \ref{surgery}, $N$
admits a metric with positive $\Gamma_2$-curvature.

We now perform surgery $\ell$-times on the manifold $N$ such that each
surgery is of codimension $n-1\geq 5$, killing in succession the
elements $r_{1},r_{2},...,r_{\ell}$. Again, according to Theorem
\ref{surgery}, the resulting manifold $M$ has fundamental group
$\pi_1(M)\cong \pi$ and admits a metric with positive
$\Gamma_2$-curvature, as desired.
\end{proof}
\subsection{Existence of metrics with positive $\Gamma_2$-curvature}
Let $M$ be a $3$-connected manifold. In particular, $M$ has a
canonical spin-structure. We use a standard notation $p_i(M)$ for the 
Pontryagin classes of $M$. There are two cases to consider here:
\begin{enumerate}
\item[(1)] The manifold $M$ is not string, i.e., $\frac{1}{2}p_1(M)\neq 0$.
\item[(2)] The manifold $M$ is string, i.e., $\frac{1}{2}p_1(M)=0$.
\end{enumerate}
In case (1), the manifold $M$ is spin and it determines a
cobordims class $[M]\in \Omega^{\spin}_n$, where $\Omega^{\spin}_n$ is
the $\spin$-cobordims group. We recall also that there is the
homomorphism $\alpha: \Omega^{\spin}_n\to KO_n$ evaluating the index
of the Dirac operator.  It is well-known that a simply connected
$\spin$-manifold $M$ of dimension at least five admits a metric with
positive scalar curvature if and only if $\alpha([M])=0$ in $KO_n$.
\begin{proof}[Proof of Theorem A] 
Let $M$ be $3$-connected, non-string (i.e. case (1) above) such that
$M$ admits a metric of positive scalar curvature. In particular, this
means that $\alpha([M])=0$. Then, according to \cite[Theorem
  B]{Stolz1}, there exists a spin cobordism between $M$ and $M'$,
where $M'$ is a total space of a geometric $\H\P^2$-bundle and has a
metric with positive $\Gamma_2$-curvature by Corollary \ref{HP2}.

We recall the following result \cite[Proposition
    3.7]{Bot-Lab}:
\begin{lemma}\label{bot-lab1}
Let $M$ be a 3-connected, non-string manifold with $\dim M\geq
9$. Assume $M$ is spin cobordant to a manifold $M'$. Then $M$ can be
obtained from $M'$ by surgeries of codimension at least five.
\end{lemma} 
Thus, we can use the surgery Theorem \ref{surgery} to ``push'' a
metric with positive $\Gamma_2$-curvature from $M'$ to $M$.
\end{proof} 
\begin{proof}[Proof of Theorem B] 
The proof is completely analogous to the arguments given to prove
\cite[Theorem B]{Bot-Lab}.
\end{proof}
If $M$ is string cobordant to zero, then the conclusion of the theorem
holds for $M$. It is known that $\Omega_n^{\str}=0$ for $n=11$ or
$n=13$; therefore any compact $3$-connected string manifold of
dimension $11$ or $13$ always has a metric with positive scalar
curvature and positive $\Gamma_2$-curvature.

\section{A general surgery theorem and applications}\label{sec:general}
\subsection{A general surgery theorem}
S.  Hoelzel proved in  \cite{Hoelzel}  an interesting general surgery theorem  which can be used to
to  provide an easy proof of  Theorem
\ref{surgery}.

Let us first introduce the general surgery theorem.  Let $\mathcal{C}_1(\Bbb{R}^n)$ denote the
vector space of curvature structures on the Euclidean space
$\mathbb{R}^n$ that satisfy the first Bianchi identity. Recall that
the orthogonal group $O(n)$ acts in a natural way on
$\mathcal{C}_1(\mathbb{R}^n)$ and that the latter space is endowed
with a canonical Euclidean inner product.

We shall say that a subset $C$ of $\mathcal{C}_1(\Bbb{R}^n)$ is
\emph{a curvature condition} if $C$ is an open convex $O(n)$-invariant
cone in $\mathcal{C}_1(\Bbb{R}^n)$.

We say that a Riemannian manifold $(M,g)$ satisfies the above
(pointwise) curvature condition $C$ if for every $p\in M$ and for any
linear isometry $i:\mathbb{R}^n \rightarrow T_pM$, the pull back by
$i$ of the Riemann curvature tensor $R$ of $(M,g)$ belongs to the
subset $C$.
\begin{theorem}\label{Hoel}
{\rm (Hoelzel, \cite[Theorem A]{Hoelzel})}
Let $C\subset \mathcal{C}_1(\Bbb{R}^n)$ be a curvature condition that
is satisfied by the standard Riemannian product metric on
$S^{c-1}\times \mathbb{R}^{n-c+1}$ for some $c$, $3\leq c\leq n$.

If a Riemannian manifold $(M^n, g)$ satisfies the curvature condition
$C$, then so does any manifold obtained from $M^n$ by surgeries of
codimension at least $c$.
\end{theorem}
\begin{remark}
It would be great to see Theorem \ref{Hoel} generalized for families
of metrics in the spirit of papers \cite{Walsh1,Walsh2}. This would
allow us to understand much more about the topology of the space of
metrics satisfying a curvature condition as above.
\end{remark}

\subsection{Applications} 
Recall that a curvature structure $R\in \mathcal{C}_1(\Bbb{R}^n)$
decomposes into $R=W+gA$ where $g$ denotes the Euclidean metric on
$\Bbb{R}^n$, $W$ is trace free, $A$ is a symmetric bilinear form
called the Schouten tensor, and the product $gA$ is the
Kulkarni-Nomizu product.  We shall say that the curvature structure
$R$ has positive $\Gamma_r$-curvature if the $k$-th elementary
symmetric function $\sigma_k(A)$ in the eigenvalues of $A$ is positive
for all $k$ with $1\leq k\leq r$.

\subsubsection{Proof of Theorem \ref{surgery}}
Let $C(\Gamma_2^+)$ be the subset of $\mathcal{C}_1(\Bbb{R}^n)$ consisting of
curvature structures with positive $\Gamma_2$-curvature.
It turns out that $C(\Gamma_2^+)$  is a curvature condition in
the above sense. Furthermore, one can check without difficulties that
the $\sigma_2$-curvature of the standard product $S^{c-1}\times
\mathbb{R}^{n-c+1}$ is equal to
$$
\sigma_2=\frac{(c-2)^2(c-1)}{8((n-2)^2(n-1)}(n(c-5)+4).
$$ This value is clearly positive for $c\geq 5$. The previous theorem
establishes then  the stability of positive $\Gamma_2$-curvature under surgeries
of codimension at least five. Therefore, we recover the surgery result
of Theorem \ref{surgery}.

Next, we provide below a geometric proof of the convexity of the set
$C(\Gamma_2^+)$.  

First, recall that we have the following standard orthogonal
decomposition of $\mathcal{C}_1(\Bbb{R}^n)$ into irreducible subspaces
\begin{equation}
\mathcal{C}_1(\Bbb{R}^n)=W\oplus gA_1\oplus \Bbb{R} g^2,
\end{equation}
where $W$ is the subspace consisting of trace free curvature
structures, $g$ denotes the Euclidean inner product of the Eucldean
space $ \Bbb{R}^n$, $A_1$ is the space of trace free bilinear forms on
$\Bbb{R}^n$, $gA_1=\{ga : a\in A_1\}$, the products $ga$ and $g^2=gg$
being the Kulkarni-Nomizu product of bilinear forms, and
$\Bbb{R}g^2=\{\lambda g^2:\lambda\in \Bbb{R}\}$.

The map $\sigma_2$ is then a quadratic function defined on
$\mathcal{C}_1(\Bbb{R}^n)$. With respect to the previous splitting, it
sends a curvature structure $R=\omega_2+g\omega_1+g^2\omega_0$ to
\begin{equation}
\sigma_2(R)=\sigma_2(\omega_2+g\omega_1+g^2\omega_0)=
-\frac{1}{2(n-2)}\|g\omega_1\|^2+\frac{1}{4}\|g^2\omega_0\|^2.
\end{equation}
Here the norms are the induced norms from the natural Euclidean
product on $\mathcal{C}_1(\Bbb{R}^n)$. In particular, the curvature
structures with null $\sigma_2$ form a cone in
$\mathcal{C}_1(\Bbb{R}^n)$ as in the figure below

\begin{figure}[htb!]
\begin{picture}(0,1)
\put(121,80){{\small $\sigma_2\!<\!0$}} 
\put(155,32){{\small $gA_1$}}
\put(-5,-5){{\small $W$}} 
\put(39,136){{\small ${\mathbb R} g^2$}}
\put(65,125){{\small $\sigma_2\!>\!0$}} 
\put(124,115){\vector(-1,0){10}}
\put(125,113){{\small $\sigma_2\!=\!0$}} 
\end{picture}
\includegraphics[height=2in]{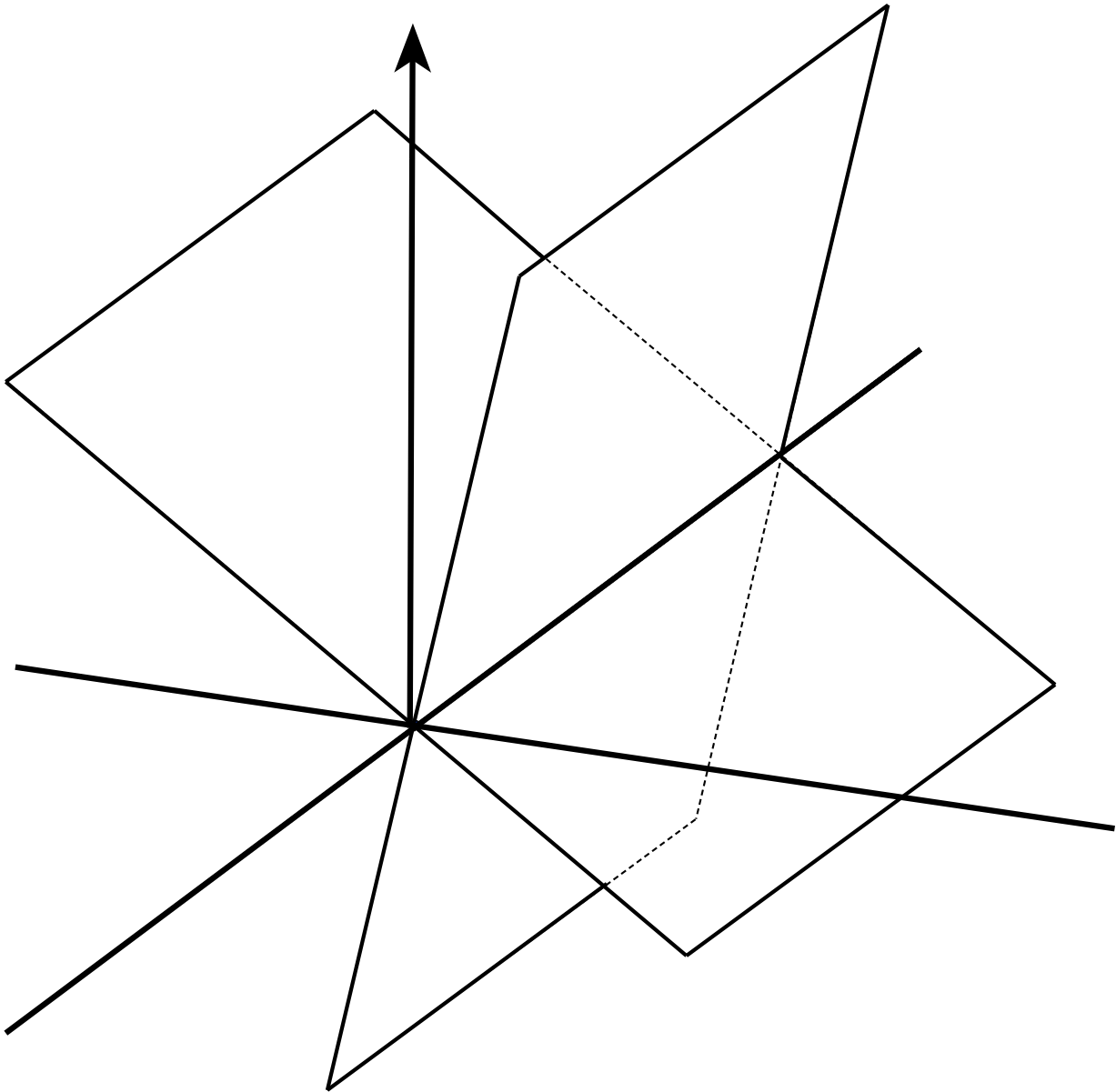}
\caption{The subsets  $\sigma_2=0,\sigma_2>0$ and $\sigma_2< 0$.}
\end{figure}
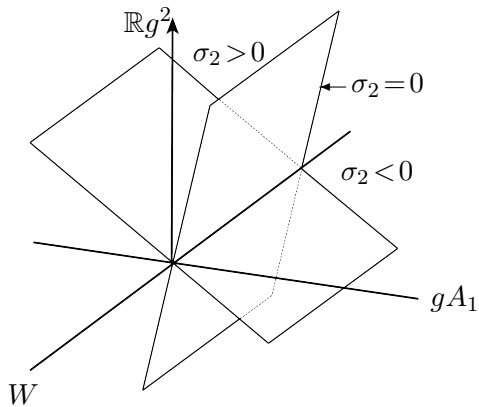\label{cone-sigma2}

Note that the curvature structures with positive $\sigma_1$ (that is
positive scalar curvature) form the upper half of the space
$\mathcal{C}_1(\Bbb{R}^n)$, that is $\{R=\omega_2+g\omega_1+\lambda
g^2:\lambda >0\}$.  It is clear now from the previous discussion that
the subset $C(\Gamma_2^+)$ is a convex cone as illustrated in the
figure below

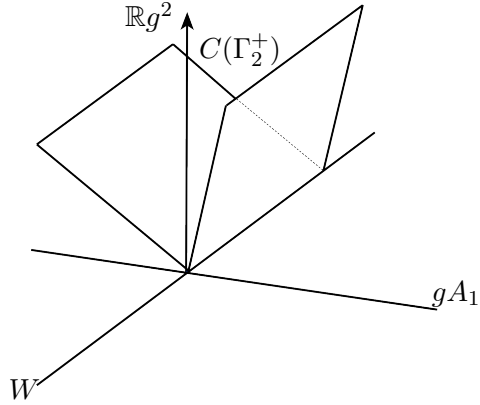
\begin{figure}[htb!]
\begin{picture}(0,1)
\put(155,32){{\small $gA_1$}}
\put(-5,-5){{\small $W$}} 
\put(39,136){{\small ${\mathbb R} g^2$}}
\put(67,124){{\small $C(\Gamma_2^+)$}} 
\end{picture}
\includegraphics[height=2in]{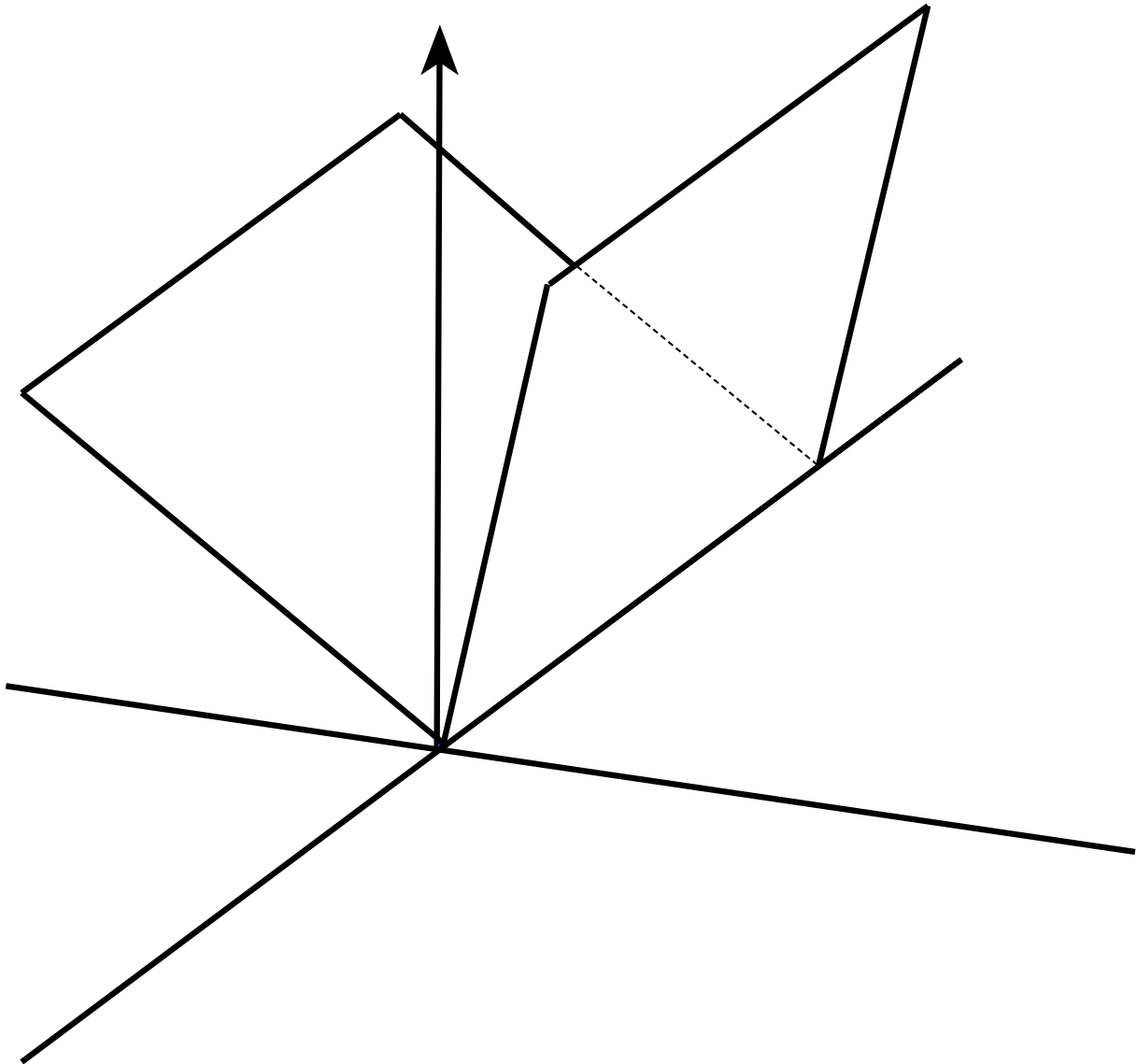}
\caption{The cone  $C(\Gamma_2^+)$.}
\end{figure}\label{cone-Gamma2}
\subsubsection{A surgery theorem for higher $\Gamma_k$-curvatures}
Now we study the  stability of higher $\sigma_k$-curvatures under
surgeries. We say that a metric $g$ has \emph{positive
  $\Gamma_k$-curvature} if $\sigma_i>0$ for all $1\leq i\leq k$.

Let $C(\Gamma_k^+)$ be the subset of $\mathcal{C}_1(\Bbb{R}^n)$
consisting of curvature structures with positive
$\Gamma_k$-curvature. It is clear that $C(\Gamma_k^+)$ is an open
subset, we are going to show that it is convex as well.

Let $R=\omega_2+g\omega_1+g^2\omega_0$ and
$\bar{R}=\bar{\omega}_2+g\bar{\omega}_1+g^2\bar{\omega}_0$ be two
curvature structures in $\mathcal{C}_1(\Bbb{R}^n)$ and $t\in [0,1]$,
then the curvature structures $(1-t)R+t\bar{R}$ splits according to
the orthogonal irreducible decomposition of $\mathcal{C}_1(\Bbb{R}^n)$
to
\[(1-t)R+t\bar{R}=((1-t)\bar{\omega}_2+t\omega_2) +
g((1-t)\bar{\omega}_1+t\omega_1)+g^2((1-t)\bar{\omega}_0+t\omega_0).\]
Consequently, we have
\begin{equation*}
\begin{array}{lll}
\bigl\{\sigma_k\left((1-t)R+t\bar{R}\right)\bigr\}^{1/k}&=&
\bigl\{\sigma_k\left( (1-t)\bar{\omega}_1+t\omega_1)
\right)\bigr\}^{1/k}
\\
\\
&\geq &  (1-t)\{\sigma_k(\bar{\omega}_1)\}^{1/k}+t\{\sigma_k(\omega_1)\}^{1/k}
\\
\\
&= &  (1-t)\{\sigma_k(\bar{R})\}^{1/k}+t\{\sigma_k(R)\}^{1/k}.
\end{array}
\end{equation*}
Here the $\sigma_k$ of a bilinear form (for instance $\omega_1$ and
$\bar{\omega}_1$) coincides with the usual symmetric function in the
eigenvalues of its corresponding operator, in particular it is well
known that $\{\sigma_k\}^{1/k}:\Bbb{R}^N\rightarrow \Bbb{R}$ is a
concave function for every $N$ and every $k\leq N$.

Consequently, the function $\{\sigma_k\}^{1/k}:
\mathcal{C}_1(\Bbb{R}^n) \rightarrow \Bbb{R}$ is concave and therefore
the subset $\{R\in \mathcal{C}_1(\Bbb{R}^n):\sigma_k(R)>0\}$ is
convex.  The subset $C(\Gamma_k^+)$ is then an intersection of convex
sets and therefore it is a convex subset as well.

We are now ready to prove the following result:
\begin{corollary}\label{connected-sum}
Let $M$ be a compact manifold with $\dim M=n \geq 2k+1$ and $M'$ be a
manifold obtained from $M$ by a surgery of codimension $n$.  If the
manifold $M$ admits a Riemannian metric with positive
$\Gamma_k$-curvature, then so does the manifold $M'$.
\end{corollary}
\begin{remark}
Corollary \ref{connected-sum} is equivalent to the stability of the
$\Gamma_k$-curvatures under connected sums. This was first proved by
Guan-Lin-Wang in \cite{GLW}.
\end{remark}
\begin{proof}[Proof of Corollary \ref{connected-sum}]
The standard Riemannian product $S^{n-1}\times \Bbb{R}$ is conformally
flat, and therefore, its $\sigma_k$-curvature coincides with the
$h_{2k}$-curvature; see Proposition \ref{conf-flat}. 

The $h_{2k}$-curvature of $S^{n-1}\times \Bbb{R}$ equals the
$h_{2k}$-curvature of $S^{n-1}$ which is positive for $n-1\geq 2k$.
The corollary follows then from Theorem \ref{Hoel}.
\end{proof}
Next we discuss the stability under $1$-dimensional surgeries of  positive
$\Gamma_k$-curvature.
\begin{corollary}\label{one-surgery}
Let $M$ be a compact manifold with $\dim M=n$, and let $k\geq 2$ be a
positive integer such that $2k<n+1-\sqrt{n-\frac{1}{n-1}}$.  Let $M'$
be a manifold obtained from $M$ by a surgery of codimension $n-1$ or
$n$.  If the manifold $M$ admits a Riemannian metric with positive
$\Gamma_k$-curvature, then so does the manifold $M'$.
\end{corollary}
\begin{proof}
Let $C({\Gamma_k^+})$ be the subset of $\mathcal{C}_1(\Bbb{R}^n)$
consisting of curvature structures with positive
$\Gamma_k$-curvature. It results from the above discussion that is a
curvature condition.

Next, we are going to show that the $\Gamma_k$-curvature of the
standard product $S^{n-2}\times \mathbb{R}^2$ is positive for $k$ as
in the corollary. Therefore, the result will follow immediately from
the above theorem of Hoelzel.

Let $\bar{A}=\frac{2(n-1)(n-2)}{n-3}A$. Then the operator $\bar{A}$
has only two distict eigenvalues: $\lambda_1=n$ with multiplicity
$n-2$, and $\lambda_2=2-n$ with multiplicity $2$. A straightforward
computation shows that the $\sigma_k$-curvature is given by
$$
\sigma_k(\bar{A})=\frac{(n-2)!n^{k-2}}{(n-k-2)!(k-2)!}\left\{\frac{n^2}{k(k-1)}-\frac{2n(n-2)}{(k-1)(n-k-1)}+\frac{(n-2)^2}{(n-k)(n-k-1)}\right\}.
$$ 
It is easy to see that the sign of $\sigma_k(\bar{A})$ is
determined by the expression:
$$
(n-1)(n^3-4kn^2+4k^2n+4k-4k^2)=(n-1)\bigl(4(n-1)k^2+4(1-n^2)k+n^3\bigr).
$$ The second factor in the previous product is quadratic in $k$ and
can easily be seen positive for $k$ and $n$ as in the corollary.
\end{proof}
As a consequence of Corollary \ref{one-surgery}, we
have:
\begin{corollary}
Let $\pi$ be a finitely presented group, and $k, n$ are arbitrary
positive integers satisfying $2\leq
k<\frac{n+1}{2}-\frac{1}{2}\sqrt{n-\frac{1}{n-1}}$.  Then there exists
a compact $n$-manifold M with positive $\Gamma_k$-curvature such that
$\pi_1(M)=\pi$.
\end{corollary}
In particular, there are no restrictions on the fundamental group of a
compact $n$-manifold of positive $\Gamma_k$-curvature in the
following cases:
\begin{itemize}
\item $n\geq 6$ and $k=2$.
\item $n\geq 8$ and $k=3$.
\item $n\geq 11$ and $k=4$.
\end{itemize}
On the other hand, a result of Guan-Viaclovsky-Wang
  \cite{GVW} asserts that positive $\Gamma_k$-curvature on an
  $n$-manifold implies positive Ricci curvature if $k\geq n/2$. In
  particular, the fundamental group of a compact $n$-manifold of
  positive $\Gamma_k$-curvature is finite provided that $k\geq n/2$.

{\bf Open Question.}  Are there any restrictions on the fundamental
group of a compact $n$-manifold of positive $\Gamma_k$-curvature if
$k\geq 2$ belongs to the following gap:
$$
\begin{array}{c}
\frac{n+1}{2}-\frac{1}{2}\sqrt{n-\frac{1}{n-1}}<k<n/2 \ ?
\end{array}
$$
\subsection{A final remark} 
In this last subsection we remark that the following surgery theorem
is not a consequence of the general surgery theorem of Hoelzel.
\begin{corollary}[\cite{Labbisgbc}]\label{h2k-posit}
Let $M$ be a compact manifold with $\dim M\geq 5$ and $M'$ be a
manifold obtained from $M$ by a surgery of codimension at least $5$.
If the manifold $M$ admits a Riemannian metric with positive second
Gauss-Bonnet curvature, then so does the manifold $M'$.
\end{corollary}
Denote by $C(h_{4}^+)$ be the subset of the space
$\mathcal{C}_1(\Bbb{R}^n)$ consisting of curvature structures with
$h_{4}>0$. We notice that the $h_{2r}$-curvature of the standard
Riemannian product $S^{c-1}\times \mathbb{R}^{n-c+1}$ is equal to the
$h_{2r}$-cuvature of the standard unit sphere $S^{c-1}$, which is
equal to $\frac {(c-1)!}{2^r (c-1 -2r)!}$ for $c-1\geq 2r$; see
\cite[Examples 2.1 and 2.2]{Labbisgbc}.  Consequently, the
$h_{2r}$-curvatures of the standard Riemannian product $S^{c-1}\times
\mathbb{R}^{n-c+1}$ are positive for $1\leq r \leq k$, provided
$c-1-2k\geq 0$.

However let us emphasize here that the set $C(h_{4}^+)$ is not a
convex subset of $\mathcal{C}_1(\Bbb{R}^n)$ (see the figure below),
and therefore Theorem \ref{Hoel} does not apply.

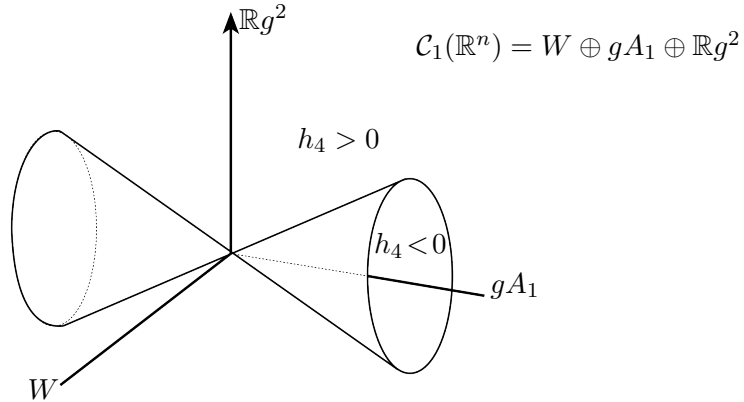
\begin{figure}[htb!]
\begin{picture}(0,1)
\put(141,50){{\small $h_4\!<\!0$}} \put(185,35){{\small $gA_1$}}
\put(10,-5){{\small $W$}} \put(90,135){{\small ${\mathbb R} g^2$}}
\put(112,90){{\small $h_4>0$}} \put(157,125){{\small
    $\mathcal{C}_1(\Bbb{R}^n)=W\oplus gA_1\oplus {\mathbb R} g^2$}}
\end{picture}
\includegraphics[height=2in]{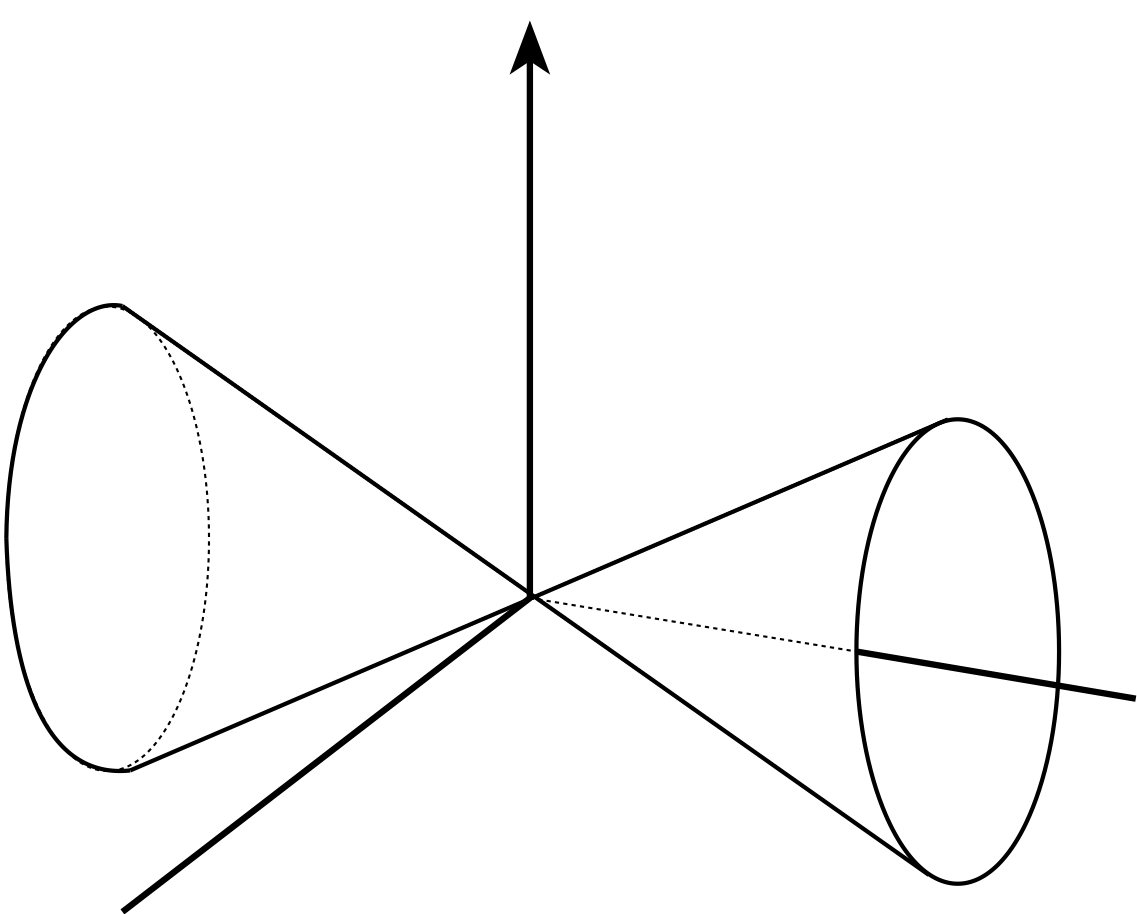}
\caption{The subsets  $h_4=0,h_4>0$ and $h_4< 0$.}\label{fig:h4}
\end{figure}

\begin{remark}
 In contrast with the cone $\sigma_2>0$, the cone $C(h_{4}^+)$
 consisting of curvature structures of positive $h_4$ has only one
 connected component, that is a connected set
\end{remark}
Hoelzel proved a more general version of the surgery theorem, namely
\cite[Theorem B]{Hoelzel}, where the convexity condition is replaced
by an \emph{inner cone condition} with respect to the standard
Riemannian curvature structure, say $S_4$, of the standard product
$S^{4}\times \mathbb{R}^{n-4}$, see \cite{Hoelzel}.

We recall that a non-empty open and $O(n)$-invariant subset $C$ of
$\mathcal{C}_1(\Bbb{R}^n)$ is said to satisfy \emph{an inner cone
  condition} with respect to $S\in \mathcal{C}_1(\Bbb{R}^n)$ if for
any $R$ in $C$, there exists a positive real number $\rho(R)$, that
depends continuously on $R$, such that
$$R+C_{\rho}:=\{R+T:T\in C_{\rho}\}\subset C,$$ where $C_{\rho}$ is an
  open convex $O(n)$-invariant cone that contains the ball
  $B_{\rho}(S)$ of radius $\rho$ in $\mathcal{C}_1(\Bbb{R}^n)$.

However, as one can realize from the figure below, the cone
$C(h_{4}^+)$ does not satisfy an inner cone condition with respect to
$S_4$.

\begin{figure}[htb!]
\begin{picture}(0,1)
\put(164,133){{\small $C_{\rho}$}}
\put(235,39){{\small $g A_1$}} \put(40,-5){{\small $W$}}
\put(130,187){{\small ${\mathbb R} g^2$}} 
\put(-17,90){{\small \begin{tabular}{|c|} \hline
        $h_4=0$\\ \hline\end{tabular}}} 
\put(60,43){{\small $R$}}
\put(80,133){{\small $R+C_{\rho}$}}
\put(25,93){\vector(1,0){13}}
\end{picture}
\includegraphics[height=2.7in]{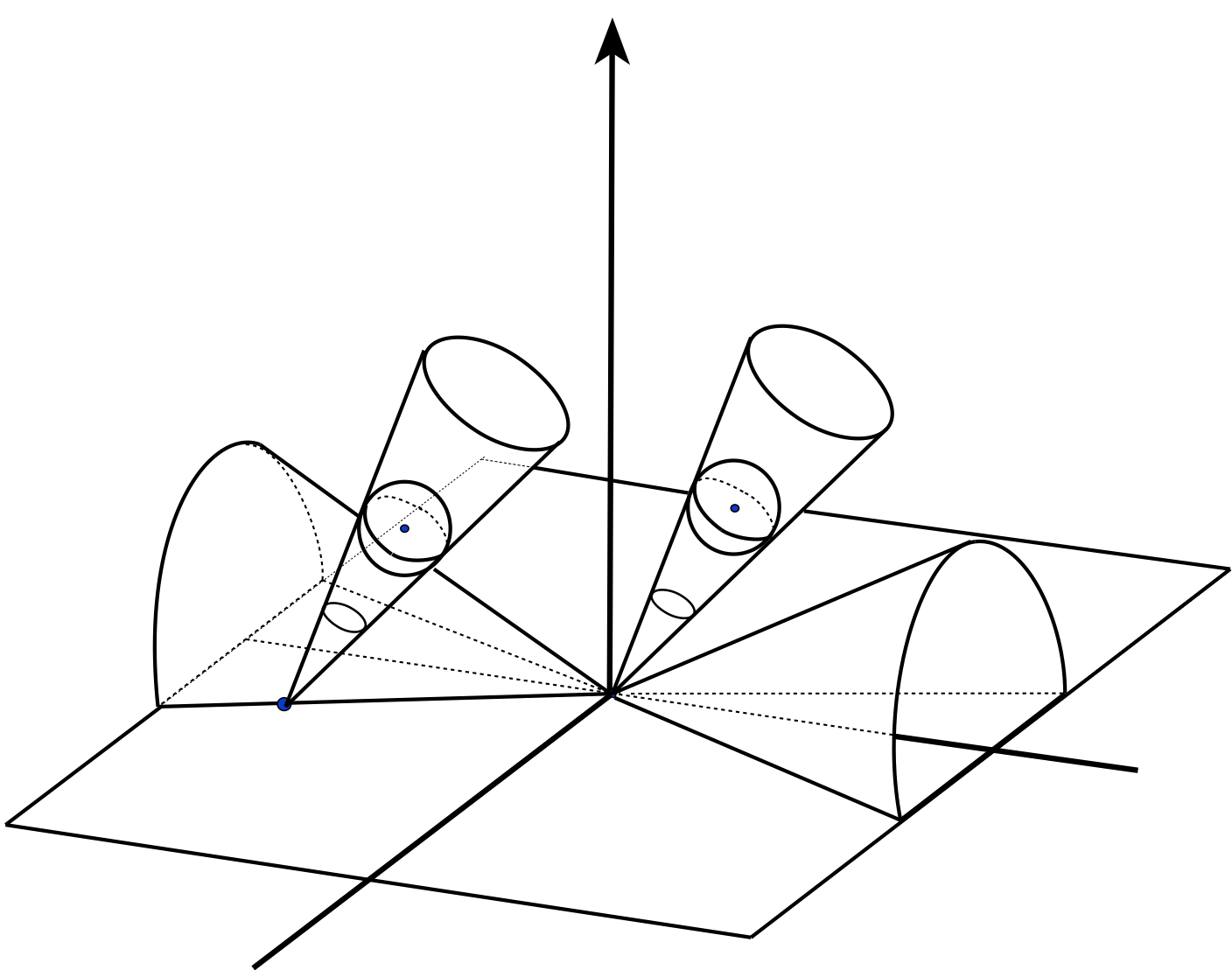}
\caption{The cone   $C(h_{4}^+)$  does not satisfy an inner cone condition.}
\label{fig:curv2}
\end{figure}
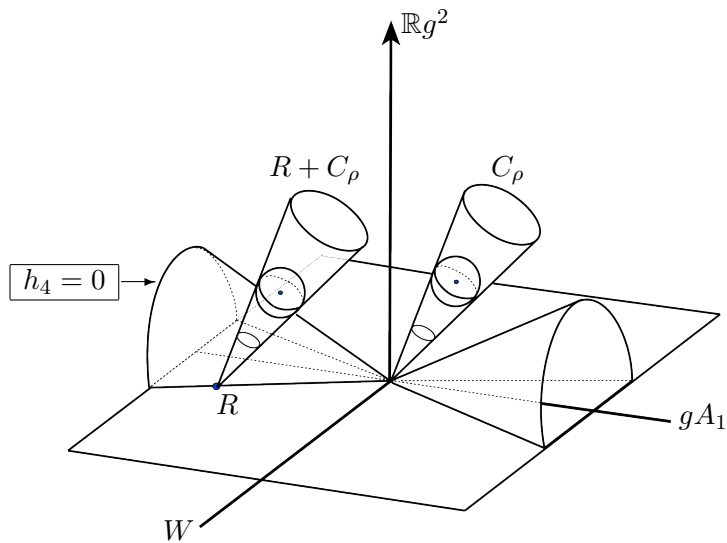


\begin{thebibliography}{9}
\bibitem{Bot-Lab} B. Botvinnik, M.L. Labbi, Highly connected
  manifolds of positive $p$-curvature, to appear in Trans. Amer. Math. Soc.
\bibitem{CNS} L. Caffarelli, L. Nirenberg, and J. Spruck, The
  Dirichlet problem for nonlinear second-order elliptic
  equations. III. Functions of the eigenvalues of the Hessian, Acta
  Math. 155 (1985), no. 3-4, 261-301.
\bibitem{CGY} S.Y.A. Chang, M. J. Gursky, and P. C. Yang, An equation
  of Monge-Ampere type in conformal geometry, and four-manifolds of
  positive Ricci curvature, Ann. of Math., 155 (2002), 709-787.
\bibitem{CHY} S.Y.A.  Chang, Z.C. Han, and P. C. Yang, Classification
  of singular radial solutions to the $\sigma_k$ Yamabe equation on
  annular domains, J. Diff. Equat. 216 (2005), no. 2,
  482-501.
\bibitem{Dessai} A.~Dessai, Some geometric properties of the Witten
  genus, Proceedings of the Third Arolla Conference on Algebraic
  Topology August 18-24, 2008.  Cont. Math. 504 (2009)  99-115.
\bibitem{GLW} P. Guan, C.S. Lin, G. Wang, Schouten tensor and some
  topological properties. Commun. Anal. Geom. 13, (2005) 887-902.
\bibitem{GVW} P. Guan, J. Viaclovsky, G. Wang, Some properties of the
  Schouten tensor and applications to conformal
  geometry. Trans. Amer. Math. Soc. 355 (2003), no. 3, 925–933.
\bibitem{GroLaw} M. Gromov, H. B. Lawson, The classification
  of simply connected manifolds of positive scalar curvature, Ann.
  of Math. 111, (1980), 423-434.
\bibitem{Hoelzel} S. Hoelzel, Surgery stable curvature condition,
  ArXiv:1303.6531.
\bibitem{Lab-ident} M.L. Labbi, On some algebraic identities and the
  exterior product of double forms, ArXiv:1112.1346.
\bibitem{Lab-min} M.L. Labbi, On $2k$-minimal submanifolds, Results
  Math. 52 (2008), no. 3-4, 323-338.
\bibitem{Lab4} M.L. Labbi,  Stability of the $p$-curvature
  positivity under surgeries and manifolds with positive Einstein
  tensor, Ann. of Global Analysis and Geometry, 15 (1997) 299-312.
\bibitem{Labbisgbc} M.L. Labbi, Manifolds with positive second
  Gauss-Bonnet curvature, Pacific Journal of Math. Vol. 227, No. 2,
  (2006), 295-310.
\bibitem{Labbih4Y} M.L. Labbi, About some quadratic scalar
  curvatures and the $h_{4}$-Yamabe equation, ArXiv:0807.2058
\bibitem{Stolz1} S.~Stolz, Simply connected manifolds of positive
  scalar curvature.  Ann.~of Math.~(2) 136 (1992), no.~3, 511--540.
\bibitem{Viac} J.  Viaclovsky, Conformal geometry and fully non linear
  equations.  Inspired by S. S. Chern, 435-460, Nankai Tracts Math.,
  11, World Sci. Publ., Hackensack, NJ, 2006.
\bibitem{Walsh1} M.~Walsh, Metrics of positive scalar curvature and
  generalised Morse functions, Part I. Mem. Amer. Math. Soc. 209
  (2011), no. 983, xviii+80 pp.
\bibitem{Walsh2} M.~Walsh, Metrics ~of ~positive ~scalar
         ~curvature and ~generalised ~Morse ~functions, ~part
         2. ArXiv:0910.2114, to appear in TAMS.
\end{thebibliography}
\end{document}